\documentclass[reqno]{amsart}
\usepackage{amsmath}
\usepackage{amssymb}
\usepackage{amsthm}
\usepackage{amsfonts}
\usepackage{mathtools}
\usepackage{bm}
\usepackage{ulem}
\usepackage{fancyhdr}
\usepackage{hyperref}
\usepackage[pdftex,final]{graphicx}
\usepackage{color}
\usepackage[iso-8859-7]{inputenc}

\theoremstyle{plain}
\newtheorem{thm}{Theorem}[section]

\newtheorem{lemma}[thm]{Lemma}
\newtheorem{prop}[thm]{Proposition}
\newtheorem*{theorem*}{Theorem}
\theoremstyle{definition}
\newtheorem{rem}[thm]{Remark}
\newtheorem{dfn}[thm]{Definition}

\numberwithin{equation}{section}

\newcommand{\mytilde}{\raise.17ex\hbox{$\scriptstyle\mathtt{\sim}$}}

\newcommand{\e}{\varepsilon} %-----------------> epsilon
\newcommand{\wk}{W^{1,p_{k}(\cdot)}(\Omega)}
\newcommand{\olgradk}{\int_{\Omega}|\nabla u_k|^{p_{k}(x)}\,dx}
\newcommand{\olgradpm}{\int_{\Omega}|\nabla u_k|^{p_{-}}\,dx}

\newcommand{\olsin}{\int_{\partial \Omega}}
\newcommand{\f}{\varphi}
\newcommand{\omk}{\overline \Omega}
\newcommand{\nk}[1]{\| {#1} \|_{1,p_{k}(\cdot)}}
\renewcommand{\r}{\rightarrow}

\newcommand{\nrm}[1]{\|{#1}\|}

\begin{document}
\title{A Neumann  problem involving the $p(x)$-Laplacian with $p=\infty$ in a subdomain}

\author{Y. Karagiorgos}
\address{Department of Mathematics, National Technical University of Athens, 
Zografou campus, Athens 15780, Greece}
\email{ykarag@math.ntua.gr}
%\ead{ykarag@math.ntua.gr} 
 %\author{N. Yannakakis\corref{cor1}}%\fnref{label1}

%\ead{nyian@math.ntua.gr} 

\author{N. Yannakakis}

%\fntext[label1]{blaa}
\address{Department of Mathematics, National Technical University of Athens, 
Zografou campus, Athens 15780, Greece}
\email{nyian@math.ntua.gr}

\begin{abstract}
In this paper we study a Neumann problem with non-homogeneous boundary condition, where the  $p(x)$-Laplacian is involved and $p=\infty$ in a subdomain. By considering a suitable sequence $p_k$ of bounded  variable exponents such that $p_k \to p$ and replacing $p$ with $p_k$ in the original problem, we prove the existence of a solution $u_k$ for each of those intermediate ones. We show that the limit of the $u_k$ exists and after giving a variational characterization of it, in the part of the domain where $p$ is bounded, we show that it is a viscosity solution in the part where $p=\infty$. Finally, we formulate the problem of which this limit function is a solution in the viscosity sense.
\end{abstract}
\keywords{Neumann problem, variable exponent, $p(x)\!-\!Laplacian$, viscosity solution, infinity Laplacian, infinity harmonic function.}

%% MSC codes here, in the form: 
\subjclass[2000]{ 35J20, 35J60, 35J70}.

\maketitle
\section{Introduction}
Consider  the following Neumann problem
\begin{equation} \label{orgproblem}
 \begin{cases} -\Delta_{p(x)}u(x)=0, & x\in \Omega  \\ |\nabla u(x)|^{p(x)-2}\frac{\partial u}{\partial n}(x)=g(x), & x \in \partial \Omega \end{cases}
\end{equation}
where  $\Omega \subset \mathbb{R}^N$ is a bounded smooth domain and $N\geq2$.
\[\Delta_{p(x)}u:=\text{div}(|\nabla u|^{p(x)-2}\nabla u)
\]  is the $p(x)$-Laplacian operator  which is the variable exponent version of the $p$-Laplacian. Also, $g \in C(\overline \Omega)$ and satisfies $\int_{\partial\Omega}g=0$. Note that this latter condition is necessary, since otherwise problem (1.1) has no solution.

The  variable exponent $p$ satisfies the following hypothesis
\begin{equation}\label{p|D=inf}
 p|_{D}=\infty
 \end{equation}
where $D$ is a compactly supported subdomain of $\Omega$,  with Lipschitz boundary.

Moreover, $p \in C^{1}(\overline{\Omega}\setminus D)$ with
\begin{equation}\label{p+cond}
 p^{+}:=\sup_{\overline{\Omega}\setminus \overline{D}} p(x)<\infty
\end{equation}
and
\vspace{-3mm}
\begin{equation}\label{p->N}
  p_{-}:=\inf_{\overline{\Omega}}p(x)>N
\end{equation}

In the literature, most of the times the variable exponent $p(\cdot)$ is assumed to be bounded. Recently, the limits $p(x)\to\infty$ have been studied in several problems where the $p(x)$-Laplacian is involved. See for instance \cite{jM10} or \cite{mP011} and the references therein. On the other hand, when $p$ is constant the limits $p\to\infty$ in problems with the $p$-Laplacian  were first studied in \cite{tB89}, in which the physical motivation was given as well. On both cases the notion of infinity Laplacian arises naturally as the limit case.

In \cite{mP011} the authors considered problem (\ref{orgproblem}) and studied the limits as $p_{n}(x)\r \infty$ uniformly in $\Omega$, where $(p_n)_n$ was a sequence of variable exponents.
J.J. Manfredi, et.al in \cite{jM09} considered condition (\ref{p|D=inf}) for the first time to study the Dirichlet problem with Lipschitz boundary conditions. To the best of our knowledge this is the first time that condition (\ref{p|D=inf}) is considered in a Neumann problem involving the $p(x)$-Laplacian.

To find out what a solution of (\ref{orgproblem}) might be, we follow the same strategy that is used  in \cite{jM09}. To be more specific we consider
a sequence of bounded variable exponents $p_k$ such that $p_k(x)=\min\{p(x), k\}$. Then $p_k(x)\r p(x)$ as $k\r\infty$, while for $k>p^+$ we have that
\begin{equation}
p_{k}(x)=\begin{cases} p(x),  & x\in \overline{\Omega}\setminus D  \\ k, & x \in D \end{cases}
\end{equation}
\begin{rem}
In \cite{jM09} the set $D$ is assumed to be convex with smooth boundary. The main reason for this is that the set of Lipcshitz function on $D$ and $W^{1,\infty}(D)$ coincide. In our case we only assume that $D$ has Lipschitz boundary which we need to have the density of smooth functions in $W^{1, p_k(\cdot)}(\Omega)$ by Proposition ~\ref{prop:density}.

\end{rem}
\begin{rem}
Note that for $k>p^+$, the boundary of the set ${\{x:p(x)>k \}}$ coincides with the boundary of $D$ and so is independent of $k$. Due to this fact we have no problems when passing to the limit as $k\to \infty$.
\end{rem}

If we replace $p$ with $p_k$ in problem (\ref{orgproblem}) we have the intermediate boundary value problems.
\[
 \begin{cases} -\Delta_{p_k(x)}u(x)=0, & x\in \Omega  \\ |\nabla u(x)|^{p_k(x)-2}\frac{\partial u}{\partial n}(x)=g(x), & x \in \partial \Omega \end{cases}
 \hspace{40mm} (1.\text{k})
\]

 Using standard methods we prove the existence of a unique weak solution $u_k$, for problem (1.k), that is also a viscosity solution.  From the Arzel\`{a}-Ascoli theorem, we then show that the  uniform limit of $(u_k)$ exists.
 We call this uniform limit $u_{\infty}$ and  show that it satisfies   a variational characterization in the set
\[
S=\bigg\{u \in W^{1, p_{-}}(\Omega):u|_{\Omega\setminus\overline{D}}\in W^{1,p(x)}(\Omega\setminus\overline{D}), \|\nabla u\|_{L^{\infty}(D)}\leq 1\, \text{and}\, \int_{\Omega}u=0  \bigg\}
\]
and that it is infinity harmonic in $D$; that is, satisfies the equation $\Delta_{\infty}u=0$, in the viscosity sense,  where
\[
\Delta_{\infty}u:= \sum_{i,j=1}^{N}\frac{\partial u}{\partial x_i}\frac{\partial u}{\partial x_j}\frac{\partial^2 u}{\partial x_{i}x_{j}}.
\]
\begin{rem} \label{rem:infinitylaplace}
Note that the infinity Laplace operator is in non-divergence form and the notion of  weak solution does not make sense in this case. To give a meaning to a solution of the equation $\Delta_{\infty}u=0$ that is not $C^2$ we need the notion of viscosity solution.
\end{rem}

\begin{rem}\label{rem:integralacond}
The condition $\int_{\Omega}u=0$ in the definition of $S$, plays a crucial role in the proof of the existence and uniqueness of the solutions $u_k$ and also in their uniform boundedness.
\end{rem}
\begin{rem}\label{rem:dirichlet}
In the Dirichlet case things are different. The existence of $u_{\infty}$ as a uniform limit of the sequence $(u_k)$  depends on the Lipschitz constant of the boundary condition and on the geometry of $D$ in $\Omega$. For reference see \cite{jM09}, \cite{jU11} and \cite{jK011}.
\end{rem}

The main results of this paper are Theorems ~\ref{thm:charact-of-u_inf} and ~\ref{thm:limitcase}. On the first one, we give a variational meaning to $u_\infty$ in $\overline\Omega\backslash D$,  where $p(\cdot)$ is bounded and  next we prove that $u_\infty$ is infinity harmonic in $D$, where $p=\infty$. On the second one, we formulate the problem (as a limit case) of which $u_\infty$ is a solution in the viscosity sense.

%The above theorems are summarized as follows.

%\begin{theorem*}
%Let $u_{\infty}:=\lim_{k\r\infty}u_k$, where $u_k$ is the unique solution of problem (1.k). If we define $I_{\infty}\colon S\to\mathbb{R}$ such that
%\[
%I_{\infty}(u):=\smashoperator{\int_{\Omega\setminus\overline{D}}}\frac{|\nabla u|^{p(x)}}{p(x)}\,dx - \smashoperator{\int_{\partial\Omega}}gu\,dS
%\]
%then $u_{\infty}$ is minimizer of $I_{\infty}$ in $S$ and is also infinity harmonic in $D$
%\end{theorem*}
%\begin{theorem*}
%$u_{\infty}$ is a viscosity solution to the following problem
%\[
%\begin{cases} -\Delta_{p(x)}u(x)=0, & x\in \Omega\backslash\overline{ D}
%\\ -\Delta_{\infty}u(x)=0, & x\in D
 %\\
 %\text{sgn}(|\nabla u(x)|-1)\text{sgn}(\frac{\partial u}{\partial \nu}(x)=0, & x \in \partial D \\
 %|\nabla u(x)|^{p(x)-2}\frac{\partial u}{\partial\nu}(x)=g(x), & x \in \partial\Omega.
 %\end{cases}
%\]
%\end{theorem*}

Partial Differential Equations involving the $p(x)$-Laplacian appear in a variety of applications.  In \cite{yC06} the authors proposed a framework for image restoration based on a variable exponent Laplacian. This was the starting point for the research on the connection between PDE's with variable exponents and image processing. Recently there has been quite a rapid progress in this direction.\footnote{The reader can visit the website \href{url}{http://www.helsinki.fi/\mytilde pharjule/varsob/index.shtml}  for further details.} Other applications that use variable exponent type Laplacians are elasticity theory and the modelling of electrorheological fluids (see \cite{mR00}).

Infinity harmonic functions (in the classical sense) were first studied by G. Arronson (see \cite{gA65, gA67}). Arronson studied the connection between infinity harmonic functions and optimal Lipschitz extensions, but only for $C^2$ functions. When the viscosity theory appeared,  Crandall, Evans and Gariepy (see \cite{mG01} or the survey paper \cite{gA04}) used viscosity solutions to prove that the connection still holds. Note that, infinity harmonic functions appear in several applications such as optimal transportation (see \cite{lE93, jC07}), image processing (see \cite{vC98}) and tug of war games (see \cite{yP09}).

\section{Preliminaries}
In this section we give some basic properties of the variable exponent Lebesgue and Sobolev spaces. For details the interested reader should refer to \cite{oK91}, \cite{xF01} and \cite{lD11}.

Let $ L^{0}(\Omega)$ be the space of real valued measurable functions in $\Omega$ and $p\colon\Omega\to[1,\infty]$ a measurable function.  We define the variable exponent Lebesgue space as
\[
L^{p(\cdot)}(\Omega)=\bigg\{u\in L^{0}(\Omega): \int_{\Omega}|\lambda\, u(x)|^{p(x)}dx <\infty,\,\, \text{for some }\,\,\lambda>0   \bigg\}
\]
equipped with the norm
\[
\|u\|_{L^{p(\cdot)}(\Omega)}:=\|u\|_{p(\cdot)}=\inf\bigg\{\lambda>0 :\int_{\Omega}\bigg|\frac{u(x)}{\lambda}\bigg| ^{p(x)}dx\leq 1 \bigg\}.
\]
The variable exponent Sobolev space  is defined by
\[
W^{1,p(\cdot)}(\Omega)=\bigg\{u\in L^{p(\cdot)}(\Omega): \nabla u \in L^{p(\cdot)}(\Omega, \mathbb{R}^N)   \bigg\}
\]
with norm
\[
\|u\|_{1,p(\cdot)}=\inf\bigg\{\lambda>0 :\int_{\Omega}\bigg|\frac{u(x)}{\lambda}\bigg|^{p(x)}+ \bigg|\frac{\nabla u(x)}{\lambda}\bigg|^{p(x)} dx\leq 1 \bigg\}.
\]
The spaces $(L^{p(\cdot)}(\Omega), \|\cdot\|_{p(\cdot)}) $, $(W^{1,p(\cdot)}(\Omega),\|\cdot\|_{1,p(\cdot)})$ are Banach spaces and if
\[
1<p_{-}:=\text{ess}\inf_{x\in\Omega}p(x)\leq p^{+}:=\text{ess}\sup_{x\in\Omega}p(x)<\infty,
\]
they are also separable and reflexive.

When p is constant, it is well known that smooth functions are dense in $W^{1,p}(\Omega)$. This is no longer true when we are dealing with the variable exponent spaces, see \cite{dE92, sS00, lD11}.  In fact, we have to consider additional conditions for the variable exponent. The most prevalent  is the so called \textit{$\log$-H\"{o}lder} continuity, i.e, there exists $C>0$, such that
\[
|p(x)-p(y)|\leq \frac{C}{\log(e + \frac{1}{|x-y|})}, \quad \text{for} \,\,x,y \in\Omega.
\]

However, it turns out that we can have the  density of smooth functions  in  some cases of discontinuous exponents (see \cite[section 9.3]{lD11}). In our case with the variable exponent $p_k$ as defined in Section 1, the following holds.
\begin{prop} \label{prop:density}
The space $C^{\infty}(\overline{\Omega})$ is dense in $W^{1,p_k(\cdot)}(\Omega)$.
\end{prop}
\begin{proof}
This is straightforward, if we use Theorem 9.3.5 of \cite[p. 298]{lD11} with\\ $\Omega_1=\Omega\backslash\overline{D}$ and $\Omega_2=D$, where each of $\Omega_i, i=1,2$ has Lipschitz boundary.
\end{proof}

\begin{prop}
Let $p\colon\Omega\to\mathbb{R}$ be a measurable function. The dual space of $(L^{p(\cdot)}(\Omega), \|\cdot\|_{p(\cdot)})$ is the space $(L^{q(\cdot)}(\Omega), \|\cdot\|_{q(\cdot)})$  where $\frac{1}{p(x)}+\frac{1}{q(x)}=1$ and the variable exponent version of H\"{o}lder inequality holds, namely
\[
\int_{\Omega}|u(x)v(x)|\,dx\leq 2 \|u\|_{p(\cdot)} \|v\|_{q(\cdot)},\quad \text{for all}\,\, u \in L^{p(\cdot)}(\Omega), v \in L^{q(x)}(\Omega).
\]
\end{prop}
%\begin{dfn}[see \cite{pH06, lD11} ]
%Let $\Omega\subset\mathbb{R}^N$ be bounded. The variable exponent $p$ satisfies the so called \textit{jump condition} in $\Omega$ if there exists $\delta>0$ such that for every $x\in\Omega$,  either
%\begin{equation}
%p^{+}_{B(x,\delta)\cap\Omega}\leq\frac{Np^{-}_{B(x,\delta)\cap\Omega}}{N-p^{-}_{B(x,\delta)\cap\Omega}}\quad \text{or}\quad  p^{-}_{B(x,\delta)\cap\Omega}\geq N.
%\end{equation}
%\end{dfn}
The next proposition is very important in the proof of the existence of a solution for problem (1.k) (see Lemma 3.2).
\begin{prop}
There exists $C>0$ such that the following Poincar\'{e} type inequality holds
\begin{equation}\label{poincare}
\|u\|_{1,p_k(\cdot)}\leq C \|\nabla u\|_{L^{p_k(\cdot)}}\,, \quad \text{for all}\,\, u \in W^{1,p_k(\cdot)}(\Omega)\,\,\text{s.t}\,\,\int_{\Omega}u=0.
\end{equation}
\end{prop}
\begin{proof}
Apply Theorem 8.2.14 in \cite[p. 256]{lD11}.
\end{proof}
\begin{rem}\label{rem:poincare}
In our case the variable exponent $ p_k(\cdot)$ for $k>p^+$ satisfies,
\begin{equation}
p_k(\cdot)\geq (p_k)_{-}\geq p_{-}>N
\end{equation}
 so  inequality (\ref{poincare}) holds and the norms $\|u\|_{1,p_k(\cdot)},\! \|\nabla\! u\|_{p_k(\cdot)}$ are equivalent in the set
\[
\bigg\{ u\in W^{1,p_k(\cdot)}(\Omega): \int_{\Omega}u(x)\,dx=0 \bigg\}.
\]
\end{rem}

\begin{prop} \label{prop:embeddings}
Let $p$ be a variable exponent such that $p_{-}>N$. Then the following holds
\begin{enumerate}
\item[(i)]
$W^{1,p(\cdot)}(\Omega)\hookrightarrow W^{1,p_{-}}(\Omega)\hookrightarrow C(\overline{\Omega})$.

\item[(ii)]
If $q \in C(\partial\Omega)$, the  embedding
\[
W^{1,p(\cdot)}(\Omega)\hookrightarrow L^{q(\cdot)}(\partial\Omega),
\]
is compact and continuous.
\end{enumerate}
\end{prop}
For reference, see  \cite{xF01, oK91, lD11} for (i) and \cite[Proposition 2.6]{jY08} for (ii).
\begin{rem}\label{rem:embeddings}
In our case, we have that $(p_k)_{-}>N$ and $p_k|_{\partial\Omega}=p \in C(\partial\Omega)$. Thus from (ii) of Proposition ~\ref{prop:embeddings}, we have that
\[
W^{1,p_{k}(\cdot)}(\Omega)\hookrightarrow L^{p_{k}(\cdot)}(\partial\Omega)
\]
\end{rem}
\begin{prop} \label{prop:norminteg-proper}
Let $u\in L^{p(\cdot)}(\Omega)$, then we have
\begin{enumerate}
\item[(i)]
If $ \|u\|_{p(\cdot)}>1$, then
\[
\|u\|_{p(\cdot)}^{p_{-}}\leq\int_{\Omega}|u(x)|^{p(x)}dx\leq\|u\|_{p(\cdot)}^{p^{+}}.
\]
\item[(ii)]
If $ \|u\|_{p(\cdot)}<1$, then
\[
\|u\|_{p(\cdot)}^{p^{+}}\leq\int_{\Omega}|u(x)|^{p(x)}dx\leq\|u\|_{p(\cdot)}^{p_{-}}.
\]
\item[(iii)]
$\|u\|_{p(\cdot)}=1 \Leftrightarrow \int_{\Omega}|u(x)|^{p(x)}dx=1$.
\end{enumerate}
\end{prop}

\hspace{5mm}
Next, we give the definition of a viscosity solution for problem (1.k).  For the general case see \cite{mC92} and \cite{jM09}. 

For a point  $x\in \partial D$ we define the set of outward unit normals $N(x)$, as the collection of all vectors $\nu$ for which we can find a sequence $(x_k)$ in 
$\partial D$, such that $x_k\to x$ and for each $k$ there exists a unique outward unit normal vector $\nu_k$ on $\partial D$ at $x_k$, such that $\nu_k\to\nu$. Note that since $D$ has Lipschitz boundary $N(x)$ is nonempty.

\begin{dfn}\label{dfn:viscosity}

\begin{enumerate}
\item[(i)]
Let $u$ be a lower semicontinuous function in $\overline{\Omega}$. We say that $u$ is a viscosity supersolution of the problem (1.k), if for every $\f\in C^2(\omk)$ such that $ u-\f $ attains its strict minimum at $x_0 \in\omk$ with $u(x_0)=\f(x_0)$, we have
\\
 \begin{itemize}
\item
if $x_0 \in \Omega\backslash\overline{ D} $, then $-\Delta_{p(x_0)}\f(x_0)\geq 0$.
\\
\item
 If $x_0 \in D $, then $-\Delta_{k}\f(x_0)\geq 0$.
\\
 \item
 If $x_0 \in \partial D  $, then
\begin{align*}
\max\{& -\Delta_{p(x_0)}\f(x_0), -\Delta_{k}\f(x_0),\\
&\sup_{\nu\in N(x_0)}\{(|\nabla \f(x_0)|^{k-2}-|\nabla \f(x_0)|^{p(x_0)-2})\nabla\f(x_0)\cdot\nu\}\}\geq 0\,.
\end{align*}
 \item
If $x_0 \in\partial\Omega $, then
\begin{align*}
\max\{|\nabla \f(x_0)|^{p(x_0)-2}\frac{\partial\f}{\partial\nu}(x_0)-g(x_0), -\Delta_{p(x_0)}\f(x_0)\}\geq 0.
\end{align*}
\end{itemize}

\item[(ii)]
Let $u$ be an upper semicontinuous function in $\overline{\Omega}$. We say that $u$ is a viscosity subsolution of the problem (1.k),  if for every $\f\in C^2(\omk)$ such that $ u-\f $ attains its strict maximum at $x_0 \in\omk$ with $u(x_0)=\f(x_0)$, we have
\\
 \begin{itemize}
\item
if $x_0 \in \Omega\backslash\overline{ D} $, then $-\Delta_{p(x_0)}\f(x_0)\leq 0$.
\\
\item
 If $x_0 \in D $, then $-\Delta_{k}\f(x_0)\leq 0$.
\\
 \item
 If $x_0 \in \partial D  $, then
\begin{align*}
\min\{& -\Delta_{p(x_0)}\f(x_0), -\Delta_{k}\f(x_0),\\
&\inf_{\nu\in N(x_0)}\{(|\nabla \f(x_0)|^{k-2}-|\nabla \f(x_0)|^{p(x_0)-2})\nabla\f(x_0)\cdot\nu\}\}\leq 0.
\end{align*}
 \item
If $x_0 \in\partial\Omega $, then
\begin{align*}
\min\{|\nabla \f(x_0)|^{p(x_0)-2}\frac{\partial\f}{\partial\nu}(x_0)-g(x_0), -\Delta_{p(x_0)}\f(x_0)\}\leq 0.
\end{align*}
\end{itemize}

\item[(iii)]
Finally, $u$ is called a viscosity solution of the problem (1.k), if it is both a viscosity subsolution and a viscosity supersolution.
\end{enumerate}
\end{dfn}

\section{The Variational and Viscosity solutions of the Intermediate Problems}
%$Let, $\Omega \subset \mathbb{R}^N$ be a bounded smooth domain, $g \in C(\overline \Omega)$ .  We study the following Neumann problem.
%\begin{enumerate}
%\item[(1.k)]
%$ \begin{cases} -\Delta_{p_{k}(x)}u(x)=0, & x\in \Omega  \\ |\nabla u(x)|^{p_{k}(x)-2}\frac{\partial u}{\partial n}(x)=g(x), & x \in \partial \Omega \end{cases} $
%\end{enumerate}
%Where for $ k > p^{+}$, $p_{k}(x)=\min\{p(x),k\}$ and $p_{k}(x) \rightarrow  p(x) \quad \forall x \in \overline{\Omega} $
\begin{dfn}
Let $u \in W^{1,p_{k}(\cdot)}(\Omega)$. We say that $u$ is a weak solution of problem (1.k) if
\begin{equation} \label{weakformulation}
\int_{\Omega}|\nabla u|^{p_{k}(x)-2}\nabla u\cdot\nabla v dx = \int_{\partial \Omega} gv dS, \quad \text{for all}\,\, v \in W^{1,p_{k}(\cdot)}(\Omega)\,.
\end{equation}

\end{dfn}
\begin{lemma} \label{lemma:existence}
There exists a unique weak solution $u_{k}$ to problem (1.k), which is the unique minimizer of the functional
\[
I_{k}(u)=\int_{\Omega}\frac{|\nabla u|^{p_{k}(x)}}{p_{k}(x)}\,dx - \int_{\partial \Omega}gu\,dS
\]
in the set
\[
S_{k}=\bigg\{u \in \wk : \int_{\Omega}u=0\bigg\}.
\]
\end{lemma}
\begin{proof}
First we show that $I_k$ is coercive and weakly lower semicontinuous, so attains its minimum in $S_k$. Let $\nk{u}\r\infty$. To obtain coercivity we need to show that $I_{k}(u)\r\infty$. Due to the fact that the norms $\nk{u}, \|\nabla u \|_{p_k(\cdot)}$ are equivalent in $S_k$,
we may suppose that $\|\nabla u \|_{p_k(\cdot)}>1$. From the  $\e$-Young inequality, the embeddings $W^{1,p_{-}}(\Omega)\hookrightarrow L^{p_{-}}(\partial \Omega)$, $W^{1,p_{k}(\cdot)}(\Omega)\hookrightarrow W^{1,p_{-}}(\Omega) $  and (i) of  Proposition~\ref{prop:norminteg-proper} we have that,

\begin{align*}
I_{k}(u) & = \int_{\Omega}\frac{|\nabla u|^{p_{k}(x)}}{p_{k}(x)}\,dx - \int_{\partial\Omega}gu\,dS \\
        & \geq\frac{1}{p_{+}}\int_{\Omega}|\nabla u|^{p_{k}(x)}\,dx -\frac{1}{\e^{p_{-}'}}\nrm{g}_{L^{p_{-}'}(\partial\Omega)}^{p_{-}'}-\e^{p_{-}}\nrm{u}^{p^-}_{L^{p_{-}}(\partial\Omega)} \\
       & \geq \frac{1}{p_+}\nk{\nabla u}^{p_{-}}-\frac{1}{\e^{p_{-}'}}\nrm{g}_{L^{p_{-}}}^{p_{-}'}-                                                  C\e^{p_{-}}\nrm{u}_{1,{p_{-}}}^{p_{-}} \\
                    & \geq \tilde{C}\nk{u}^{p_{-}} -\frac{1}{\e^{p_{-}'}}\nrm{g}_{L^{p_{-}'}}^{p_{-}'}-C\e^{p_{-}}\nk{u}^{p_{-}} \\
         & \geq \nk{u}^{p_{-}}(\tilde{C}-C\e^{p_{-}})-\frac{1}{\e^{p_{-}'}}\nrm{g}_{L^{p_{-}'}}^{p_{-}'},
\end{align*}

where $p'_{-}$ is the conjugate exponent of $p_{-}$.

If we choose $\e>0$ small enough such that $\tilde{C}-C \e^{p_{-}}>0$,  we have that $I_{k}(u)\r\infty$ as $\nk{u}\r\infty$. Thus, $I_k$ is coercive.

For the weak lower semicontinuity let $u_n\xrightarrow[]{w} u$ in $S_k$. Using the weak lower semicontinuity of the integral $\int_{\Omega}\frac{|\nabla u|^{p_{k}(x)}}{p_{k}(x)}$ and the embedding $S_k\hookrightarrow L^{p_{-}}(\partial\Omega)$ we obtain that $ I_k$ is weak lower semicontinuous. Hence, $I_k$ attains its minimum in $S_k$. The uniqueness is standard due to the strict convexity of $I_k$. It remains to show that the unique minimizer $u$ is also a unique weak solution of problem (1.k). Let $v\in W^{1,p_k(\cdot)}(\Omega)$ and set $\tilde{v}=v-\frac{1}{|\Omega|}\int_{\Omega}vdx$. Then $\tilde{v}\in S_k$ and using the fact that $u$ minimizes $I_k$ in $S_k$, it is easy to see that $u$ satisfies (\ref{weakformulation}) and hence is a weak solution of problem (1.k). Due to the fact that a weak solution of problem (1.k) is a minimizer of $I_k$ in $S_k$, the proof is completed.

\end{proof}
The next Lemma is very useful since it provides us with a problem that has the same weak solutions as problem (1.k) but which allows us to take separate cases.

\begin{lemma}\label{lemma:equivproblem}
The following problem
\begin{equation}\label{equivproblem}
\begin{cases} -\Delta_{p_k(x)}u(x)=0, & x\in \Omega\backslash\overline{ D}
\\ -\Delta_{k}u(x)=0, & x\in D
 \\
 |\nabla u(x)|^{k-2}\frac{\partial u}{\partial \nu}(x)=|\nabla u(x)|^{p(x)-2}\frac{\partial u}{\partial \nu}(x), & x \in \partial D \\
 |\nabla u(x)|^{p(x)-2}\frac{\partial u}{\partial\nu}(x)=g(x),& x \in \partial\Omega.
 \end{cases}
\end{equation}
has the same weak solutions as problem (1.k).
 \end{lemma}
\begin{proof}
Let $k>p^+$, then $C^{\infty}(\overline{\Omega})$ is dense in $W^{1,p_k(\cdot)}(\Omega)$ (see Proposition ~\ref{prop:density}). If we take as a test function $v\in C^{\infty}(\overline{\Omega})$,  use integration by parts, Gauss-Green theorem and the fact that $D$ is compactly supported in $\Omega$, we conclude that the weak formulation of (\ref{equivproblem}) is (\ref{weakformulation}).
\end{proof}
The next Lemma is crucial in the proof of our main results. Note that the importance of the condition $\int_{\Omega}u=0$ is evident.
\begin{lemma}\label{lemma:equic-unifbndness}
Let $u_k$ be a weak solution of  problem (1.k). Then the sequence $(u_k)$ is equicontinuous and uniformly bounded.
\end{lemma}
\begin{proof}
If we multiply (1.k) by $u_k$ and use integration by parts, we obtain
\[
\olgradk = \olsin g u_{k}\,dS\leq 2 \|u_{k}\|_{L^{p_{k}(\cdot)}(\partial\Omega)} \|g\|_{L^{q_{k}(\cdot)}(\partial\Omega)} \leq C(\Omega,g) \|\nabla u_{k}\|_{L^{p_{k}(\cdot)}}
\]
where we used the variable exponent version of H\"{o}lder's inequality, (\ref{poincare}) and the embedding $\wk \hookrightarrow L^{p_{k}(\cdot)}(\partial \Omega)$ (see Remark ~\ref{rem:embeddings}). We consider the cases
\begin{itemize}
\item
if $\|\nabla u_{k}\|_{L^{p_{k}(\cdot)}} \leq 1$, then $\olgradk \leq 2C(\Omega,g)$
\vspace{2mm}
\item
if $\|\nabla u_{k}\|_{L^{p_{k}(\cdot)}} >1$, then from (i) of Proposition ~\ref{prop:norminteg-proper}, we have
\begin{align*}
\olgradk & \leq C(\Omega,g)\|\nabla u_{k}\|_{L^{p_{k}(\cdot)}}=C\big(\|\nabla u_{k}\|_{L^{p_{k}(\cdot)}}^{p_{-}}\big)^{\frac{1}{p_{-}}}
\\ & \leq C\bigg(\olgradk\bigg)^{\frac{1}{p_{-}}}.
\end{align*}
So, we end up with  the following inequality
\begin{equation}\label{p_k-integral-estimate}
\olgradk \leq C(\Omega,g,p_{-}),
\end{equation}
where C is independent of $k$.
\end{itemize}

On the other hand, since $p_{-}>N$ from Morrey's inequality (see \cite[p. 183]{lE10}) we have
\[
|u_{k}(x)-u_{k}(y)|\leq C(N,p_{-})|x-y|^{1-\frac{N}{p_{-}}}\bigg(\olgradpm\bigg)^{\frac{1}{p_{-}}}, \quad \text{for all}\,\, x, y \in \overline{\Omega}
\]
and
\[
\olgradpm=\smashoperator{\int_{\Omega\cap\{|\nabla u_k|\leq 1  \}}}|\nabla u_{k}|^{p_{-}}\,dx + \smashoperator{\int_{\Omega\cap\{|\nabla u_k| > 1  \}}}|\nabla u_{k}|^{p_{-}}\,dx
\]
\begin{equation}\label{p-integestimate}
\leq C(\Omega,p_{-}) + \olgradk \leq C(\Omega,p_{-},g),
\end{equation}
where in the last inequality we used the previous estimate for $\olgradk$. From the above, we obtain that
\begin{equation}\label{equicont}
|u_{k}(x)-u_{k}(y)|\leq C(\Omega,N,p_{-},g)|x-y|^{1-\frac{N}{p_{-}}}, \quad \text{for all}\,\, x, y \in \overline{\Omega}.
\end{equation}
Hence, the sequence $(u_k)$ is equicontinuous in $C(\overline{\Omega})$.

It remains to show that the sequence $(u_k)$ is uniformly bounded in $\overline{\Omega}$.
Let $k>p^{+}$. Since we are assuming that $\int_{\Omega}u_{k}=0$ and $u_{k} \in C(\overline{\Omega}) $, we may choose a
point $y \in \Omega$ such that $u_{k}(y)=0$. Then, from (\ref{equicont}) we get
\begin{align*}
|u_{k}(x)|& \leq C(\Omega,N,p_{-},g)|x-y|^{1-\frac{N}{p_{-}}}\\
 &\leq C(\Omega,N,p_{-},g)(\text{diam}(\Omega))^{1-\frac{N}{p_{-}}}\leq C(\Omega,N,p_{-},g)
\end{align*}
so $(u_k)$ is uniformly bounded in $\overline{\Omega}$ and this concludes the proof.
\end{proof}
\begin{rem}
Note  that in the above Lemma, in contrast to \cite[Proposition 2.5]{jM09} there can be no comparison of the $p_{-}$ norm of $|\nabla u_k|$ with $I_k(u_k)$, due to the existence of the term $\int_{\partial \Omega}gu_k dS$ in the definition of $I_k$. To overcome this difficulty we had to use the estimates given in Proposition ~\ref{prop:norminteg-proper}.

\end{rem}

\begin{prop}\label{prop:weaktoviscosity}
Let $u$ be a continuous weak solution of (1.k). Then $u$ is a solution of (1.k) in the viscosity sense.
\end{prop}
\begin{proof}
We prove that $u$ is a viscosity supersolution of (1.k). The proof that $u$ is a viscosity subsolution is similar.
Let $x_0 \in \omk$ and $\f \in C^2(\omk)$ such that $u-\f$ attains its strict minimum at $x_0$ and $(u-\f)(x_0)=0$.
We want so show that $-\Delta_{p_k(x_0)}\f(x_0)\geq 0$. To argue by contradiction suppose that $-\Delta_{p_k(x_0)}\f(x_0)<0$.
We consider the following cases.
\begin{itemize}
\item
Let $x_0 \in \Omega\backslash \overline{D} $. Then $-\Delta_{p(x_0)}\f(x_0)<0$ and by continuity there exists $r>0$ such that $B_r(x_0)\subset \Omega\backslash \overline{D}$ and for every $x\in B_r(x_0)$ we have
\begin{align*}
-\Delta_{p(x)}\f(x) =&-|\nabla \f(x)|^{p(x)-2}\Delta \f(x) \\
                   = &-(p(x)-2)|\nabla\f(x)|^{p(x)-4}\Delta_{\infty}\f(x)\\
                  & -|\nabla \f(x)|^{p(x)-2}\ln(|\nabla \f(x)|)\nabla \f(x)\cdot\nabla p(x)<0.
\end{align*}
Set \[
m=\inf_{x\in S(x_0,r)}(u-\f)(x) >0\,\, \text{and}\,\, \tilde{\f}=\f +\frac{m}{2}.\]
Then $\tilde{\f} $ satisfies $\tilde{\f}(x_0)>u(x_0)$ and  $\tilde{\f}(x)\leq u(x)$,  for every $x\in S(x_0,r)$.
Moreover,
\[
-\Delta_{p(x)}\tilde{\f}(x)<0, \quad \text{for all}\,\, x \in B_r(x_0).
\]
Multiplying by $(\tilde{\f}-u)^+$  and integrating we get,
\begin{align*}
0>&\smashoperator{\int_{B_r(x_0)}}|\nabla \tilde{\f}(x)|^{p(x)-2}\nabla \tilde{\f}(x)\cdot\nabla(\tilde{\f}-u)^+dx \\
  = &\smashoperator{\int_{B_r(x_0)\cap \{\tilde{\f}>u\}}}|\nabla \tilde{\f}(x)|^{p(x)-2}\nabla \tilde{\f}(x)\cdot \nabla(\tilde{\f}-u)dx.
\end{align*}
 If we extend $(\tilde{\f}-u)^+$ as zero outside of $B_r(x_0)$ and use it as a test function in the weak formulation
of $-\Delta_{p(x)}u(x)=0$, we obtain
\begin{align*}
0=&\smashoperator{\int_{\Omega\backslash\overline{ D}}}|\nabla u(x)|^{p(x)-2}\nabla u(x)\cdot\nabla(\tilde{\f}-u)^+dx \\
  = &\smashoperator{\int_{B_r(x_0)\cap \{\tilde{\f}>u\}}}|\nabla u(x)|^{p(x)-2}\nabla u(x)\cdot \nabla(\tilde{\f}-u)\,dx.
\end{align*}
By subtracting and using a well known inequality (see \cite[p.51]{jK011}) we conclude
\begin{align*}
0>&\smashoperator{\int_{B_r(x_0)\cap \{\tilde{\f}>u\}}}(|\nabla \tilde{\f}(x)|^{p(x)-2}\nabla \tilde{\f}(x)-|\nabla u(x)|^{p(x)-2}\nabla \tilde{u}(x))\cdot\nabla(\tilde{\f}-u)\,dx\\
  \geq & c\smashoperator{\int_{B_r(x_0)\cap \{\tilde{\f}>u\}}}|\nabla \tilde{\f}-\nabla u|^{p(x)}dx,
\end{align*}

which is a contradiction.
\item
If $x_0\in D$ then the proof is exactly the same, so $-\Delta_{k}\f(x_0)\geq 0$.
\vspace{2mm}
\item
%%%%%%%%%%%%%%%%%%%%%%%%%%%%%%%%%%%%%%%%%%%%%%%%%%%%%%
%%%%%%%%%%%%%%%%%%%%%%%%%%%%%%%%%%%%%%%%%%%%%%%%%%%%%%change
Let $x_0\in \partial D$. Since $\partial D$ is not smooth the normal vector field $\nu(\cdot)$ is not uniquely defined. In particular to each $x\in\partial D$ there corresponds a set of outward unit normals $N(x)$. Following \cite{Barles} we need to show that
\begin{align*}
\max\{& -\Delta_{p(x_0)}\f(x_0), -\Delta_{k}\f(x_0),\\
&\sup_{\nu\in N(x_0)}\{(|\nabla \f(x_0)|^{k-2}-|\nabla \f(x_0)|^{p(x_0)-2})\nabla\f(x_0)\cdot\nu\}\}\geq 0\,.
\end{align*}
We argue again by contradiction. So, by continuity there exists $r>0$ such that
\begin{align*}
&-\Delta_{p(x)}\f(x)<0, \quad \text{for all}\,\, x \in B_r(x_0)\cap\Omega\backslash\overline{ D}\\
&-\Delta_{k}\f(x)<0, \quad \text{for all}\,\, x \in B_r(x_0)\cap D\,.
\end{align*}

and
\[
(|\nabla \f(x)|^{k-2}-|\nabla \f(x)|^{p(x)-2})\nabla\f(x)\cdot\nu<0\,,
\]
for all $x \in
B_r(x_0)\cap \partial D$ and all $\nu \in
N(x)$.
%%%%%%%%%%%%%%%%%%%%%%%%%%%%%%%%%%%%%%%%%%%%%%%%%%%%%%%%%%%%%%%%change
Set
\[
m=\inf_{x\in S(x_0,r)}(u-\f)(x) >0\,\, \text{and}\,\, \tilde{\f}=\f + \frac{m}{2}.
\]
Then $\tilde{\f} $ satisfies $\tilde{\f}(x_0)>u(x_0)$ and  $\tilde{\f}(x)\leq u(x)$,  for every $x\in S(x_0,r)$. Multiplying the first two inequalities
by $(\tilde{\f}-u)^+$, integrating by parts, adding and using the third one, we have
\begin{align*}
\smashoperator{\int_{B_r(x_0)\cap(\Omega\backslash\overline D) }}|\nabla & \tilde{\f}|^{p(x)-2}\nabla \tilde{\f}\cdot\nabla(\tilde{\f}-u)^+ dx
+ \smashoperator{\int_{B_r(x_0)\cap D}}|\nabla \tilde{\f}|^{k-2}\nabla \tilde{\f}\cdot\nabla(\tilde{\f}-u)^+ dx  \\
& <\smashoperator{\int_{B_r(x_0)\cap\partial D}}(|\nabla \tilde{\f}|^{k-2}-|\nabla \tilde{\f}|^{p(x)-2})\frac{\partial\f}{\partial\nu}(\tilde{\f}-u)^+ dS<0.
\end{align*}
On the other hand, we may extend $(\tilde{\f}-u)^+$ as zero outside $B_r(x_0)$, take it as a test function in the weak formulation of (1.k) and reach a contradiction as we did in the previous case.
\item
Let $x_0 \in \partial\Omega$. We need to show that
\[
\max\{|\nabla \f(x_0)|^{p(x_0)-2}\frac{\partial\f}{\partial\nu}(x_0)-g(x_0), -\Delta_{p(x_0)}\f(x_0)\}\geq 0.
\]
To contradiction, suppose that
\[
|\nabla \f(x_0)|^{p(x_0)-2}\frac{\partial\f}{\partial\nu}(x_0)-g(x_0)<0
\]
and
\[
-\Delta_{p(x_0)}\f(x_0)<0.
\]
Proceeding as before, we get
\[
\smashoperator{\int_{B_r(x_0)\cap\{\tilde{\f}>u\}}}|\nabla \tilde{\f}|^{p(x)-2}\nabla\tilde{\f}\cdot\nabla (\tilde{\f}-u)\,dx <\smashoperator{\int_{\partial\Omega\cap B_r(x_0)\cap\{\tilde{\f}>u\}}}(\tilde{\f}-u)g\,dS
\]
and
\[
\smashoperator{\int_{B_r(x_0)\cap\{\tilde{\f}>u\}}}|\nabla u|^{p(x)-2}\nabla u \cdot\nabla (\tilde{\f}-u)\,dx =\smashoperator{\int_{\partial\Omega\cap B_r(x_0)\cap \{\tilde{\f}>u\}}}(\tilde{\f}-u)g\,dS
\]
which is a contradiction. Thus, $u$ is a viscosity supersolution of (1.k).
 \end{itemize}
\end{proof}
\begin{rem}
In  \cite{jM09} the case $x_0 \in \partial \Omega$ is trivially verified because of the Dirichlet boundary condition. In our case (Proposition ~\ref{prop:weaktoviscosity} and Theorem ~\ref{thm:limitcase}) the boundary condition in the viscosity sense is
not immediately satisfied and one has to use the continuity of the boundary data $g$.
\end{rem}

\section{Passing to the limit }

Consider the following set
\[
S=\bigg\{u \in W^{1, p_{-}}(\Omega):u|_{\Omega\setminus\overline{D}}\in W^{1,p(\cdot)}(\Omega\backslash\overline{D}),\, \|\nabla u\|_{L^{\infty}(D)}\leq 1\, \text{and}\, \int_{\Omega}u=0  \bigg\}.
\]

If $v \in S \subset S_{k}$, we have
\begin{align*}
I_{k}(v)&=\smashoperator{\int_{\Omega\setminus\overline{D}}}\frac{|\nabla v|^{p(x)}}{p(x)}\,dx + \smashoperator{\int_{D}}\frac{|\nabla v|^{k}}{k}\,dx - \smashoperator{\int_{\partial \Omega}}gv\,dS\\
        & \leq \smashoperator{\int_{\Omega\setminus\overline{D}}}\frac{|\nabla v|^{p(x)}}{p(x)}\,dx +\frac{|D|}{k} - \smashoperator{\int_{\partial \Omega}}gv\,dS
\end{align*}

and passing to the limit as $k\r\infty$,  we have
\[
\liminf_{k}I_{k}(v) \leq \smashoperator{\int_{\Omega\setminus\overline{D}}}\frac{|\nabla v|^{p(x)}}{p(x)}\,dx - \smashoperator{\int_{\partial\Omega}}gv\,dS :=I_{\infty}(v)
\]

The next theorem is the first main result of this paper. We give a variational characterization of the limit function $u_\infty$ in $\Omega\setminus\overline{D}$, where $p^{+}:=\sup_{\overline{\Omega}\setminus \overline{D}} p(x)<\infty$ and next we prove that $u_\infty$ is infinity harmonic in $D$ in the viscosity sense.

\begin{thm}\label{thm:charact-of-u_inf}
Let $u_k$ be the unique minimizer of $I_k$ in $S_{k}$. Then there exists a function $u_{\infty} \in S$, such that $u_{\infty}$ minimizes  $I_{\infty}$ in $S$ and is also infinity harmonic in $D$.
\end{thm}
\begin{proof}

From Lemma ~\ref{lemma:equic-unifbndness} and the Arzel\`{a} - Ascoli theorem, there exists a subsequence of $(u_k)$ (denoted again $u_k$) and a function   $u_{\infty} \in C(\overline\Omega)$ such that \[u_k\r u_{\infty}\,,\quad \text{uniformly in}\quad \overline\Omega.\] First we show that $u_\infty\in S $. From the estimate (\ref{p-integestimate}) and the Poincar\'{e}-Wirtinger inequality in $W^{1,p_{-}}(\Omega)$ (recall that $\int_{\Omega}u_k=0$) we have that $(u_k)$ is bounded in $W^{1,p_{-}}(\Omega)$. Thus, \[u_k\xrightarrow[]{w} u_{\infty}\,,\quad \text{in}\quad W^{1,p_{-}}(\Omega).\] To obtain that  $u_\infty\in W^{1,p(\cdot)}(\Omega\setminus\overline D)$ we use the estimate (\ref{p_k-integral-estimate}) for the integral $\olgradk$  and the inequality (\ref{poincare}), to show that $(u_k)$ is bounded in  $W^{1,p(\cdot)}(\Omega\backslash\overline D)$. Thus, \[u_k \xrightarrow[]{w} u_{\infty}\,,\quad \text{in}\quad W^{1,p(\cdot)}(\Omega\backslash\overline D).\] Where we used again the pointwise convergence of $(u_k)$ to $u_{\infty}$ in $\overline\Omega$. Now let $m>p_{-}$ and $k>m$. Then from
H\"{o}lder's inequality and (\ref{p_k-integral-estimate}) we have
\begin{align*}
\bigg(\smashoperator{\int_{D}}|\nabla u_k|^{m}dx\bigg)^{\frac{1}{m}}&\leq |D|^{\frac{1}{m}-\frac{1}{k}}\bigg(\smashoperator{\int_{D}}|\nabla u_k|^{k}dx\bigg)^{\frac{1}{k}}\\
&\leq|D|^{\frac{1}{m}-\frac{1}{k}}\bigg(\smashoperator{\int_{\Omega}}|\nabla u_k|^{p_{k}(x)}dx\bigg)^{\frac{1}{k}}\leq |D|^{\frac{1}{m}-\frac{1}{k}}(C(\Omega,g))^{\frac{1}{k}}.
\end{align*}
Take $k$ large enough such that $|D|^{\frac{1}{m}-\frac{1}{k}}(C(\Omega,g))^{\frac{1}{k}}\leq2|D|^{\frac{1}{m}}$ holds. Then we have
\begin{equation}\label{m-norm-gradient-estimate}
\|\nabla u_k\|_{L^{m}(D)}\leq2|D|^{\frac{1}{m}}.
\end{equation}
From (\ref{m-norm-gradient-estimate}) and the Poincar\'{e}-Wirtinger inequality we have that $(u_k)$ is bounded in $W^{1,m}(D)$. This fact together with the pointwise convergence of $(u_k)$ to $u_{\infty}$, gives that  $u_k \xrightarrow[]{w} u_{\infty}$ in $W^{1,m}(D)$.

 Let $m>p^{+}$. From the weak lower semicontinuity of the integral and the H\"{o}lder inequality we have
\begin{align*}
\|\nabla u_{\infty}\|_{L^{m}(D,\mathbb{R}^{N})} &\leq\liminf_{k}\bigg(\smashoperator{\int_{D}}|\nabla u_k|^{m}dx\bigg)^{\frac{1}{m}}\\
&\leq\liminf_{k}\big[|D|^{\frac{1}{m}-\frac{1}{k}}\bigg(\smashoperator{\int_{D}}|\nabla u_k|^{k}dx\bigg)^{\frac{1}{k}}\big]\\
&\leq\liminf_{k}\big[|D|^{\frac{1}{m}-\frac{1}{k}}\bigg(C(\Omega,g)\bigg)^{\frac{1}{k}}\big]\\
&=|D|^{\frac{1}{m}}.
\end{align*}
Thus, passing to the limit as $m\r \infty$, we have $\|\nabla u_{\infty}\|_{L^{\infty}(D,\mathbb{R}^{N})}\leq 1$. The condition $\int_{\Omega}u_{\infty}=0$ is immediately satisfied since $\int_{\Omega}u_{k}=0$ for each $k$. Thus, $u_{\infty} \in S$.

It remains to show that $u_\infty$ minimizes $I_{\infty}$ in $S$. To this end, let  $v \in S$. Then by the minimizing property of $u_k$ in $S_k$ and the weak lower semicontinuity of $I_{\infty}$ we have
\begin{align*}
I_{\infty}(u_{\infty})&\leq\liminf_{k}I_{\infty}(u_k)\leq\liminf_{k}I_{k}(u_k)\\
&\leq\liminf_{k}I_{k}(v)\leq I_{k}(v).
\end{align*}
Thus, $u_{\infty}$ minimizes $I_{\infty}$ in $S$. To prove that $u_\infty$ is infinity harmonic in $D$ we use the fact that $u_k$ is $k$-harmonic in $D$ and $u_k\to u_\infty$ uniformly in $\overline{\Omega}$ (see \cite[Proposition 2.2]{tB89}, \cite{rJ93} or Theorem 2.8 in \cite[p. 17]{jK011}).
\end{proof}
%%%%%%%%%%%%%%%%%%%%%%%%%%%%%%%%%%%%%%%%%%%%%%%%%%%%%%%%%%%%%%%%%%%%%%%%%%%%%%%%%%%%%%%%%%%%%%%
%%%%%%%%%%%%%%%%%%%%%%%%%%%%%%%%%%%%%%%%%%%%%%%%%%%%%%%%%%%%%%%%%%%%%%%%%%%%%%%%%%%%%%%%%%%%%%%change
\begin{rem}
The minimizer of $I_{\infty}$ in $S$ is unique. Indeed let $u_1, u_2$ be minimizers of $I_{\infty}$ in $S$ that are also infinity harmonic in $D$. Then, there exists some $C\in\mathbb{R}$ such that
\[
u_2= u_1 + C \quad\text{in}\quad \overline\Omega\backslash D,
\]
which by uniqueness of the Dirichlet problem for the infinity Laplacian gives us that $u_2= u_1 + C$ also in $D$. Hence $u_2= u_1 + C$ in $\Omega$ which implies uniqueness by the fact that the mean value of both $u_1$, $u_2$ is zero.
\end{rem}
%%%%%%%%%%%%%%%%%%%%%%%%%%%%%%%%%%%%%%%%%%%%%%%%%%%%%%%%%%%%%%%%%%%%%%%%%%%%%%%%%%%%%%%%%%%%%%%%%
In the following theorem we state the problem of which $u_\infty$ is a viscosity solution. This arises naturally from Lemma ~\ref{lemma:equivproblem} and Proposition ~\ref{prop:weaktoviscosity} as a limit case.
\begin{thm}\label{thm:limitcase}
Let $(u_k)$ be the sequence of solutions of  problems (1.k) and $u_{\infty}$ the uniform limit
of a subsequence of  $(u_k)$. Then $u_\infty$ is a viscosity solution of the following problem

\[
\begin{cases} -\Delta_{p(x)}u(x)=0, & x\in \Omega\backslash\overline{ D}
\\ -\Delta_{\infty}u(x)=0, & x\in D
 \\
 \textup{sgn}(|\nabla u(x)|-1)\textup{sgn}(\frac{\partial u}{\partial \nu}(x))=0, & x \in \partial D \\
 |\nabla u(x)|^{p(x)-2}\frac{\partial u}{\partial\nu}(x)=g(x), & x \in \partial\Omega
 \end{cases}
\]
\end{thm}
\begin{proof}
Let $x_0 \in \overline{\Omega}$, $\f \in C^2(\omk)$ such that $u_{\infty}-\f$ attains its strict minimum at $x_0$. We will show that $u_{\infty}$ is a viscosity supersolution. The case of viscosity subsolution is exactly the same, so we omit the proof.  We consider the cases.\begin{itemize}
\item
Let $x_0 \in \Omega\backslash\overline{ D}$. Since $u_k \r u_{\infty}$ uniformly in $ \overline{\Omega}$, we can find a sequence $(x_k)$ in $\Omega\backslash\overline{ D}$ such that $x_k \r x_0$ and $u_k-\f$ attains its strict minimum at $x_k$ (see for instance \cite[Proposition 2.2]{tB89} or Theorem 2.8 in \cite[p. 17]{jK011}). Since $u_k$ is a viscosity solution of (1.k), we have that ${-\Delta_{p(x_k)}\f(x_k)\geq 0}$, for all $k>p^+$ and passing to the limit as $k\r\infty$ we obtain that $-\Delta_{p(x_0)}\f(x_0)\geq 0$.

\item
Let $x_0 \in D$. As before, we can find points $x_k$ in $D$ such that $x_k \r x_0$ and $u_k-\f$ attains its strict minimum at $x_k$. Since $u_k$ is a viscosity solution of (1.k), we have that
\[
-(|\nabla \f(x_k)|^{k-2}\Delta\f(x_k)+(k-2)|\nabla \f(x_k)|^{k-4}\Delta_{\infty}\f(x_k))\geq 0.
\]
If $\nabla\f(x_0)=0$, then $\Delta_{\infty}\f(x_0)=0$ and there is nothing to prove. Assume now that $\nabla \f(x_0)\neq 0$.  Then $\nabla\f(x_k)=0$, for large $k$ and  dividing the previous inequality with $(k-2)|\nabla \f(x_k)|^{k-4}$, we obtain

\[									
-\frac{|\nabla \f(x_k)|^{2}\Delta\f(x_k)}{k-2}-\Delta_{\infty}(x_k)\geq 0.
\]
Passing to the limit as, $k\r\infty$ we get
\[
-\Delta_{\infty}\f(x_0)\geq 0.
\]
Thus,  $u_{\infty}$ is viscosity supersolution of $-\Delta_{\infty}u=0$ in $D$.
\item
%%%%%%%%%%%%%%%%%%%%%%%%%%%%%%%%%%%%%%%%%%%%%%%%%%%%%%%%%%%%%%%%%%%%%%%%%%%%%%%%%change
Let $x_0 \in \partial D$. Again following \cite{Barles} we need to show that,
\begin{align*}
\max\{& \Delta_{p(x_0)}\f(x_0), \Delta_{\infty}\f(x_0),\\
&\sup_{\nu\in N(x_0)}\{\text{sgn}(|\nabla \f(x_0)|-1)\text{sgn}(\nabla\f(x_0)\cdot\nu)\}\}\geq 0.
\end{align*}
Due to uniform convergence again we can find a sequence $(x_k)$ such that $x_k \r x_0$ and $u_k-\f$ attains its strict minimum at $x_k$. We consider the following cases.
\begin{itemize}
\item[(i)]
If infinitely many $x_k$ belong to $\Omega\backslash\overline{ D}$, then for large $k$, we have \[-\Delta_{p(x_k)}\f(x_k)\geq 0\]  and passing to the limit as $k\r\infty$, we obtain that
\[-\Delta_{p(x_0)}\f(x_0)\geq 0.\]
\item[(ii)]
If infinitely many $x_k$ belong to $D$, then for large $k$, we have \[-\Delta_{k}\f(x_k)\geq 0\] and proceeding as we did in the second case, we have $ -\Delta_{\infty}\f(x_0)\geq 0$.
\item[(iii)]
If infinitely many $x_k$ belong to $\partial D$, then from Proposition ~\ref{prop:weaktoviscosity} we get
\[
\sup_{\nu\in N(x_k)}\{(|\nabla \f(x_k)|^{k-2}-|\nabla \f(x_k)|^{p(x_k)-2})\nabla\f(x_k)\cdot\nu\}\geq 0.
\]
Hence,
\[
\sup_{\nu\in N(x_k)}\{(|\nabla \f(x_k)|^{k-p(x_k)}-1)\nabla\f(x_k)\cdot\nu\}\geq 0.
\]
This implies that
\[
\sup_{\nu\in N(x_0)}\{(|\nabla \f(x_0)|-1)\nabla\f(x_0)\cdot\nu\}\geq 0.
\]
But this is the case when
\[
\sup_{\nu\in N(x_0)}\{\text{sgn}(|\nabla \f(x_0)|-1)\text{sgn}(\nabla\f(x_0)\cdot\nu)\}\geq 0.
\]
\end{itemize}
\item
Let $x_0 \in \partial\Omega$.  We want to prove, that
\[
\max\{|\nabla \f(x_0)|^{p(x_0)-2}\frac{\partial\f}{\partial\nu}(x_0)-g(x_0), -\Delta_{p(x_0)}\f(x_0)\}\geq 0
\]
Again, we can find points $x_k$ in $\omk$ such that $x_k\r x_0$ and $u_{k}-\f$ has a strict minimum at $x_k$. If infinitely many $x_k$ belong to $\Omega$, then  $-\Delta_{p(x_k)}\f(x_k)\geq 0$ and passing to the limit as $k\r\infty$, we get $-\Delta_{p(x_0)}\f(x_0)\geq 0$. If infinitely many $x_k$ belong to $\partial\Omega$, from Proposition ~\ref{prop:weaktoviscosity}, we have
\[
\max\{|\nabla \f(x_k)|^{p(x_k)-2}\frac{\partial\f}{\partial\nu}(x_k)-g(x_k), -\Delta_{p(x_k)}\f(x_k)\}\geq 0.
\]
If $-\Delta_{p(x_k)}\f(x_k)\geq 0$, as before we conclude that \[-\Delta_{p(x_0)}\f(x_0)\geq 0.\]
If, \[|\nabla \f(x_k)|^{p(x_k)-2}\frac{\partial\f}{\partial\nu}(x_k)-g(x_k)\geq 0,\] passing to the limit as $k\r\infty$ and since $g\in C(\omk)$, we conclude that \[|\nabla \f(x_0)|^{p(x_0)-2}\frac{\partial\f}{\partial\nu}(x_0)-g(x_0)\geq 0\] and this completes the proof.

\end{itemize}

\end{proof}

%\bibliography{...}
\end{document}